\documentclass[10pt, a4paper]{scrartcl}

\usepackage[T1]{fontenc}
\usepackage{lmodern}
\usepackage[english]{babel}
\usepackage{amsfonts}
\usepackage{amsmath} 
\usepackage{amssymb}
\usepackage{amsthm}
\usepackage{mathrsfs}
\usepackage[x11names,dvipsnames,svgnames]{xcolor}
\usepackage{booktabs}
\usepackage{float}
\usepackage[backgroundcolor=blue!10]{todonotes}
\usepackage{enumitem}
\usepackage{graphicx}
\usepackage{tikz,pgf}
\usetikzlibrary{calc}
\usetikzlibrary{decorations.markings}
\usetikzlibrary{shapes.geometric}
\usetikzlibrary{patterns}
\usetikzlibrary{intersections}
\usepackage{multirow}
\usepackage{array}
\usepackage{adjustbox}
\usepackage{calc}
\usepackage{hyperref}
\hypersetup{
    colorlinks,
    linkcolor={red!50!black},
    citecolor={blue!50!black},
    urlcolor={blue!80!black}
}

\usepackage{cleveref}
\usepackage{bm}

\theoremstyle{plain}
\newtheorem{lemma}{Lemma}[section]
\newtheorem{proposition}[lemma]{\textbf{Proposition}}
\newtheorem{theorem}[lemma]{\textbf{Theorem}}
\newtheorem*{theoremnonum}{\textbf{Theorem}}

\theoremstyle{definition}
\newtheorem{definition}[lemma]{\textbf{Definition}}

\newtheorem{remark}[lemma]{Remark}

\newcommand{\Q}{\mathbb{Q}}
\newcommand{\R}{\mathbb{R}}
\newcommand{\C}{\mathbb{C}}
\newcommand{\p}{\mathbb{P}}
\newcommand{\Tan}{\mathbb{T}}

\newcommand{\col}{\Psi}

\renewcommand{\tilde}{\widetilde}

\newcommand{\sprod}[2]{\langle #1,#2\rangle_{\!\Moebius}}
\newcommand{\dinf}{\omega_\infty}

\newcommand{\South}{{\bm{s}}}
\newcommand{\North}{{\bm{n}}}
\newcommand{\vertOne}{{\bm{w}_1}}
\newcommand{\vertTwo}{{\bm{w}_2}}
\newcommand{\cycle}{\mathcal{C}}
\newcommand{\variety}{Y}
\newcommand{\Moebius}{M}
\newcommand{\Fin}{\operatorname{Fin}}

\colorlet{ecol}{black!50!white}
\definecolor{colR}{rgb}{.932,.172,.172} 
\definecolor{colB}{rgb}{.255,.41,.884} 
\colorlet{colG}{Goldenrod2}

\colorlet{col1}{DarkGoldenrod}
\colorlet{col2}{DarkKhaki}
\colorlet{col3}{MediumPurple}
\colorlet{col4}{IndianRed}
\colorlet{col5}{SkyBlue!90!black}
\colorlet{col6}{MediumSeaGreen}

\tikzstyle{vertex}=[circle, draw, fill=black, inner sep=0pt, minimum size=4pt]
\tikzstyle{svertex}=[circle, draw, fill=black, inner sep=0pt, minimum size=2pt] 
\tikzstyle{edge}=[line width=1.5pt,ecol]
\tikzstyle{redge}=[edge,colR]
\tikzstyle{bedge}=[edge,colB]

\tikzstyle{labelsty}=[font=\scriptsize]

\tikzstyle{dedge}=[
	edge,%
	postaction={%
		decoration={markings,
			mark=at position 0.5
			with {
				\node[dart,fill,inner sep=1pt,minimum size=6pt,transform shape] (0,0) {};	
			},
		},%
	decorate}%
]

\tikzstyle{sdedge}=[
	edge,%
	postaction={%
		decoration={markings,
			mark=at position 0.5
			with {
				\node[dart,fill,inner sep=1pt,minimum size=4pt,transform shape] (0,0) {};
			},
		},%
	decorate}%
]

\title{Zero-sum cycles in flexible polyhedra}
\date{}

\author{%
Matteo Gallet$^{\diamond}$%
\and
Georg Grasegger$^{\ast,\triangleright}$%
\and
Jan Legersk\'y$^{\circ}$%
\and
Josef Schicho$^{\ast, \circ}$
}
\renewcommand{\thefootnote}{\fnsymbol{footnote}}

\begin{document}
\maketitle
\footnotetext{\hspace{0.15cm}$^\ast$ Supported by the Austrian Science Fund (FWF): W1214-N15, project DK9.\\%
$^\circ$ Supported by the Austrian Science Fund (FWF): P31061.\\%
$^\diamond$ Supported by the Austrian Science Fund (FWF): Erwin Schr\"odinger Fellowship J4253.\\%
$^\triangleright$ Supported by the Austrian Science Fund (FWF): P31888.\\%
This is the accepted version of the
following article: \textit{Matteo Gallet, Georg Grasegger, Jan Legersk\'y, and Josef Schicho.
Zero-sum cycles in flexible polyhedra. Bulletin of the London Mathematical Society. 54(1):112--125, 2022},
which has been published in final form at \href{https://doi.org/10.1112/blms.12562}{doi:10.1112/blms.12562}.
}
\begin{abstract}
 We show that if a polyhedron in the three-dimensional affine space with triangular faces is flexible, 
 i.\,e., can be continuously deformed preserving the shape of its faces,
 then there is a cycle of edges
 whose lengths sum up to zero once suitably weighted by~$1$ and~$-1$.
 We do this via elementary combinatorial considerations,
 made possible by a well-known compactification of the three-dimensional affine space
 as a quadric in the four-dimensional projective space.
 The compactification is related to the Euclidean metric,
 and allows us to use a simple degeneration technique that 
 reduces the problem to its one-dimensional analogue, which is trivial to solve.
\end{abstract}
\renewcommand{\thefootnote}{\arabic{footnote}}

\section*{Introduction}

Flexibility of polyhedra --- namely, the existence of a continuous deformation
preserving the shapes of all faces --- is a well-studied topic, 
a limit case of the theory of rigidity of surfaces 
(investigated by, among others, Cohen-Vossen, Nirenberg, Alexandrov, Gluck, see~\cite{Ghomi2020} for an overview) 
that allows for combinatorial and topological or algebro-geometric techniques. 
Although flexible polyhedra have been studied for long, most of their aspects are mysterious.
With the notable exception of Bricard octahedra (see~\cite{Bricard1897}), 
very few families are known and classified, 
and even natural questions like the existence of not self-intersecting flexible polyhedra 
are non-trivial to answer (see~\cite{Connelly1977, Kuiper1979}).
Therefore, providing necessary conditions for flexibility of polyhedra
seems a relevant step towards a better understanding of these objects.

The starting point of the theory of rigidity and flexibility of polyhedra
can be considered to be the work by Cauchy~\cite{Cauchy1813}, 
in which he proved that convex polyhedra are rigid.
Gluck~\cite{Gluck1975} then showed that ``almost all'' simply connected
polyhedra are rigid as well. 
The research then focused on finding necessary and sufficient conditions
for the flexibility of polyhedra, and it is still an active research field,
as witnessed by, among others, recent works by Gaifullin~\cite{Gaifullin2018}
and Alexandrov~\cite{Alexandrov2019, Alexandrov2020}.
In the latter, Alexandrov proved that 
the edge lengths of a flexible (orientable) polyhedron with triangular faces must satisfy a $\Q$-linear relation.

The goal of this paper is to prove that for every flexible polyhedron in the three-dimensional affine space with triangular faces
there is a cycle of edges and a sign assignment such that the sum of the signed edge lengths is zero.
Notice that we do not ask the polyhedron to be homeomorphic to a sphere, neither to be embedded nor immersed.
This is a generalization of the analogous statement that holds for \emph{suspensions},
i.\,e., simplicial complexes that have the combinatorics of a double pyramid;
see \cite{Connelly1974, Mikhalev2001, Alexandrov2011}.
The first statement of this kind appeared, to our knowledge,
in~\cite{Lebesgue1967} concerning flexible octahedra (see Figure~\ref{figure:bricard}).
In a previous work~\cite{Gallet2020} we re-prove this statement 
in the case of flexible octahedra by means of symbolic computation.
There, however, our method provides a finer control on which edges 
take the positive sign, and which the negative sign.
Here, instead, our result is more general, but gives no information 
about which edges are counted as positive and which as negative.
\begin{figure}[ht]
\centering
 \tikz[baseline={(0,0)}]{\node at (0,0) {\includegraphics[width=.37\textwidth, trim = {0 1.1cm 0 1.1cm}, clip]{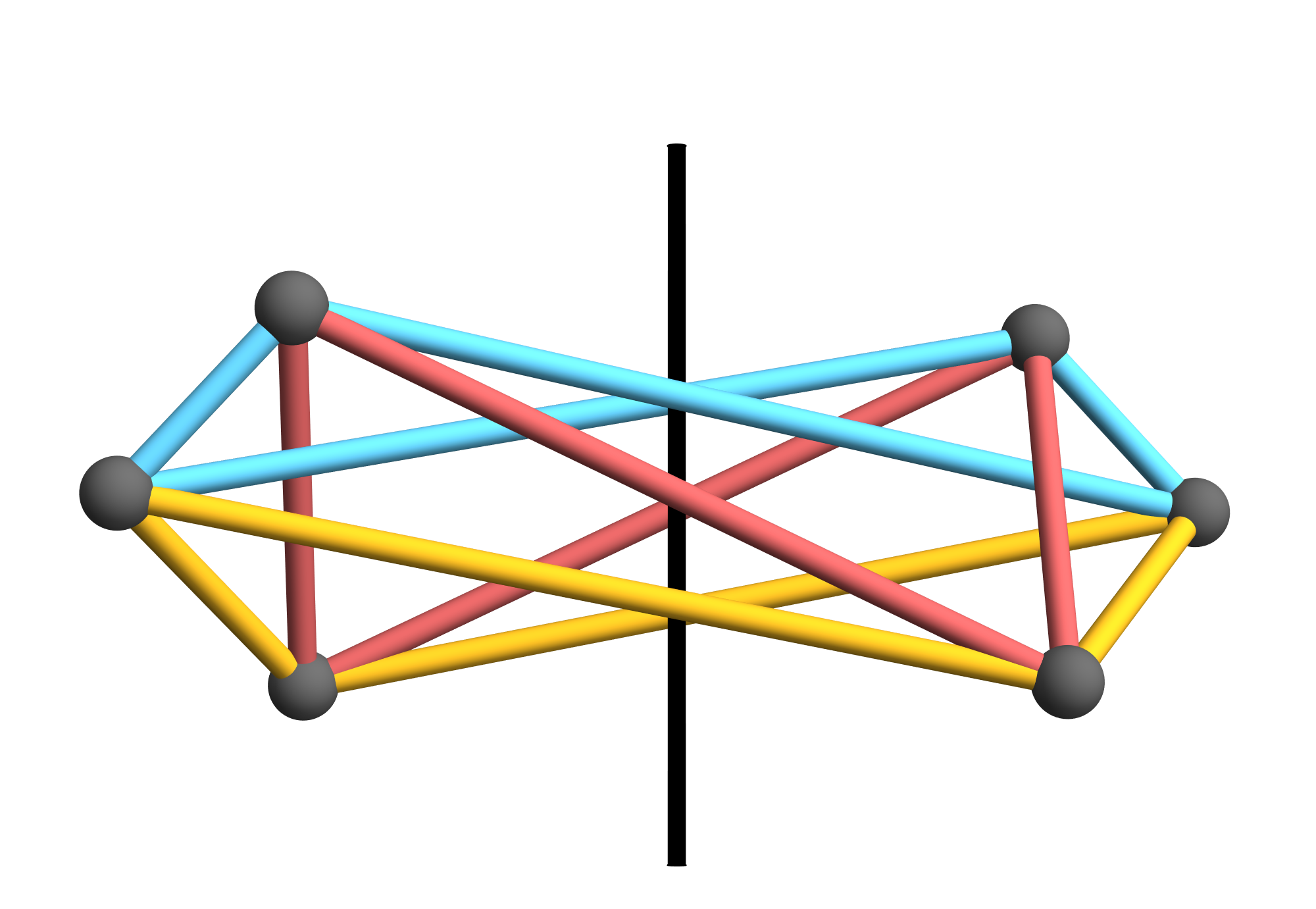}}}
 \tikz[baseline={(0,0)}]{\node[rotate=90] at (0,0) {\includegraphics[width=.3\textwidth, trim = {0 0.5cm 0 0.5cm}, clip]{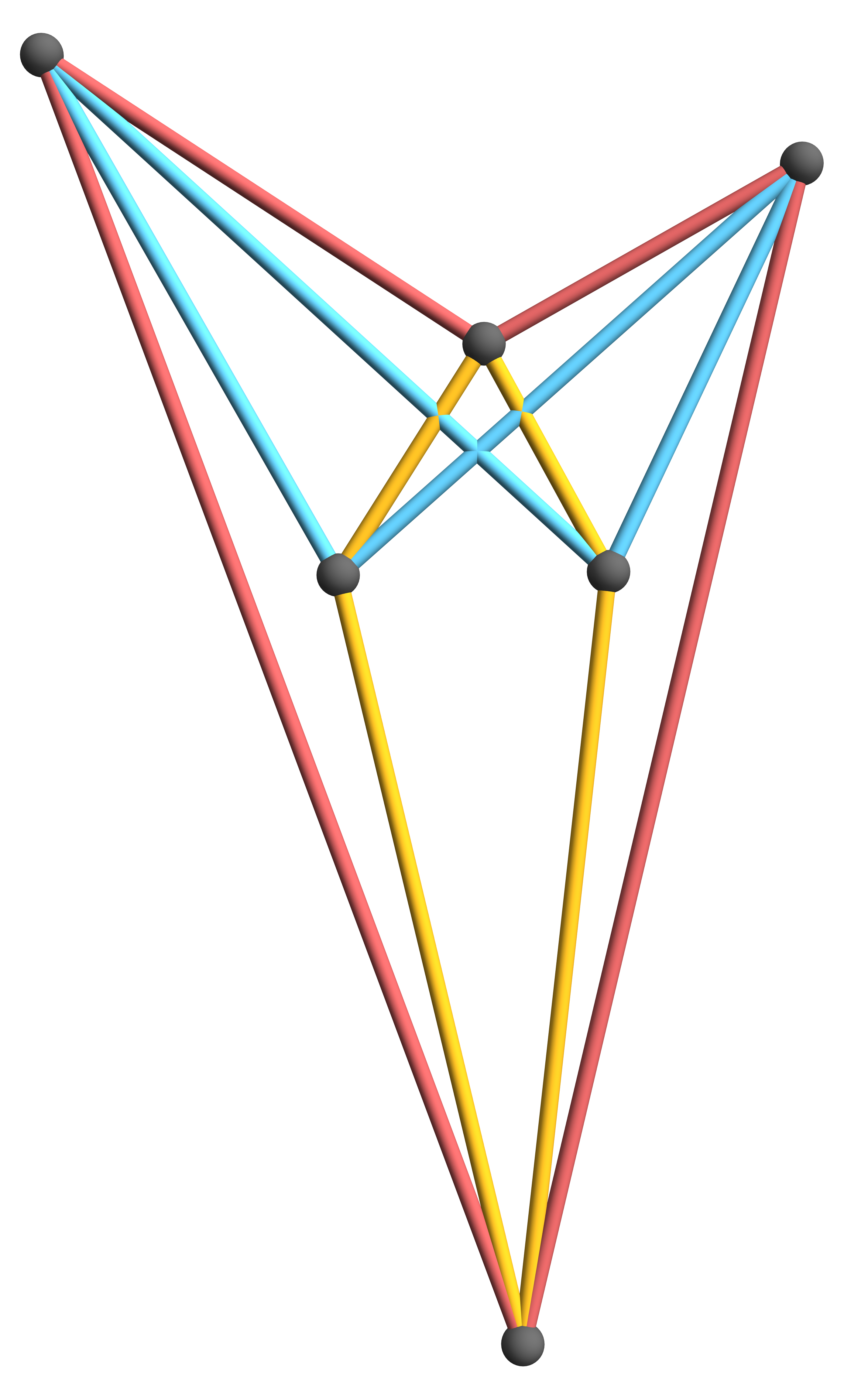}}}
 \caption{Flexible octahedra and the cycles indicated by colors in which the signed lengths of edges sum up to zero.
 The one on the left is a so-called Type I flexible octahedron, which is symmetric with respect to a line,
 while the one on the right is a so-called Type III flexible octahedron, which admits two flat positions,
 one of which is depicted here.}
 \label{figure:bricard}
\end{figure}

This is the main result of our paper (\Cref{theorem:main}):

\begin{theoremnonum}
    Consider a polyhedron with triangular faces that admits a flex,
    i.\,e., a continuous deformation preserving the shapes of all faces.
	Let $\{ \vertOne, \vertTwo \}$ be an edge, and let $\South$ and $\North$ be the two vertices adjacent to both~$\vertOne$ and~$\vertTwo$. 
	If the dihedral angle between the faces $\{\vertOne,\vertTwo, \South \}$ and $\{\vertOne,\vertTwo, \North\}$ is not constant along the flex,
	then there is an induced cycle of edges containing $\{ \vertOne, \vertTwo \}$ but neither the vertex~$\South$ nor~$\North$ 
	and there is a sign assignment such that the signed sum of
	lengths of the edges in the cycle is zero.
\end{theoremnonum}

To understand the core argument of the proof, consider a one-dimensional analogue of the situation in the theorem:
we are given a cycle with vertices mapped to~$\R$, namely a sequence $x_1, \dotsc, x_k$, $x_{k+1} = x_1$ of real numbers.
The lengths of the edges of this cycle are given by $| x_i - x_{i+1} |$ for $i \in \{1, \dotsc, k\}$. 
If we multiply each edge length by $\mathrm{sign}(x_i - x_{i+1})$ we get $x_i - x_{i+1}$. 
Therefore, if we sum the edge lengths, each multiplied by $\mathrm{sign}(x_i - x_{i+1})$,
we get a telescoping sum, which yields~$0$.
What we do is to reduce the main result of the paper to this simple statement,
and this gives the proof of the theorem.
The way we perform this reduction is by considering ``limits'' of the flex of the polyhedron:
the key point is that these ``limits'' are taken in a suitable compactification of~$\R^3$;
moreover, they are given by points whose coordinates are complex numbers.
The compactification, which is a quadric hypersurface in the four-dimensional projective space,
is chosen in such a way that the usual Euclidean distance of~$\R^3$ extends to these points ``at infinity''.
A ``limit'' of the flex then determines a coloring of the vertices of the polyhedron.
Combinatorial properties of this coloring are obtained from the algebro-geometric features
of the compactification of~$\R^3$. 
These combinatorial properties yield the existence of the cycle passing through the edge~$\{ \vertOne, \vertTwo \}$ in the statement,
such that all its vertices lie, in the ``limit'' situation, on the same tangent space to the quadric hypersurface.
A direct computation shows that on such tangent space the extension of the Euclidean distance behaves
similarly to a one-dimensional distance, and the argument at the beginning of the paragraph concludes the proof.

The paper is structured as follows.
\Cref{compactification} introduces the compactification of the three-dimensional affine space
we are going to use, and proves the basic relations between the compactification and 
the standard Euclidean metric.
\Cref{result} formalizes the notion of flexibility,
sets up the combinatorial constructions that are needed for the main result,
and proves it.

\section{The M\"obius model}
\label{compactification}

We introduce a particular compactification of~$\R^3$, 
which we call the \emph{conformal compactification} (\Cref{definition:compactification}).
This means that we consider~$\R^3$ as a subset of a compact space
(in our case, a projective variety)
and so we add to~$\R^3$ some ``points at infinity'',
similarly as it is done in the construction of the projective space.
The nice feature of this compactification is that it behaves well
with the standard Euclidean distance, in a way that allows extending
this distance in a coherent way also to some of these points at infinity (\Cref{proposition:distance}).
To really use these points at infinity,
we need to extend our setting to the complex numbers.
This section then includes a series of small results regarding
various properties of this extension of the Euclidean distance,
which constitute the algebro-geometric backbone of the main result of the paper.

\begin{definition}
\label{definition:compactification}
We embed $\R^3$ into~$\p^4$ by the map 
\[
 (x,y,z) \mapsto (x:y:z:x^2+y^2+z^2:1) \,,
\]
which we call the \emph{conformal}, or \emph{M\"obius}, \emph{embedding}.
Every point $(x:y:z:r:h)$ in the image of $\R^3$ under the conformal embedding
lies on the hypersurface $\Moebius \subset \p^4$ of the equation
\[
 x^2+y^2+z^2-rh = 0 \,,
\]
which is a projective model for conformal geometry, and this is where the name of the embedding comes from.
More precisely, each and every real point on~$\Moebius$ with $h \neq 0$ corresponds to a unique point in~$\R^3$, and vice versa.
\end{definition}

Here we make precise the connection between~$\Moebius$ 
and the Euclidean distance.

\begin{definition}
The symmetric bilinear form on~$\R^5 \times \R^5$ 
associated to the equation of~$\Moebius$ is denoted by~$\sprod{\cdot}{\cdot}$,
and is given by
\[
	\sprod{(x_1,y_1,z_1,r_1,h_1)}{(x_2,y_2,z_2,r_2,h_2)} = 
	x_1 \, x_2 + y_1 \, y_2 + z_1 \, z_2 -\frac{1}{2}(r_1 \, h_2 + r_2 \, h_1)\,.
\]
In the following, we apply this bilinear form to points in~$\p^4$:
by this we mean that we choose vector representatives of the points
for the computation.
Whenever needed, we specify explicitly which representative we choose:
if $p \in \p^4$, by a slight abuse of notation whenever we write
$p = (x: y: z: r: h)$ we take the vector $(x,y,z,r,h)$ as representative of~$p$.
\end{definition}

\begin{proposition}
\label{proposition:distance}
 Let $u_1, u_2 \in \R^3$ and let $p_1, p_2$ be the corresponding points in~$\Moebius$
 under the conformal embedding, with $p_i = (x_i: y_i: z_i: r_i: 1)$. Then
 \[
  - \frac{1}{2} \, \left\| u_1 - u_2 \right\|^2 = 
  \sprod{p_1}{p_2} \,,
 \]
 where $\left\| \cdot \right\|$ is the Euclidean norm in~$\R^3$. 
\end{proposition}
\begin{proof}
 By construction, we have that $u_i = (x_i, y_i, z_i)$. 
 Then expanding both left and right hand side using the definitions 
 yields the same expression.
\end{proof}

We now start the exploration of the ``points at infinity'' of~$\Moebius$. 

\begin{definition}
\label{definition:classification_points}
The points $(x:y:z:r:h)$ in~$\Moebius$ such that $h \neq 0$ are called \emph{finite}.
The other points of~$\Moebius$, namely the ones on the quadric $\Moebius_{\infty} := \Moebius \cap \{ h = 0 \}$, 
form a cone over the plane quadric $A := \{ (x:y:z) \, | \, x^2 + y^2 + z^2 = 0 \}$ 
whose vertex is the point $\dinf := (0:0:0:1:0)$. 
We call the points in $\Moebius_{\infty} \setminus \{ \dinf \}$ \emph{simple infinite}.
\end{definition}

From the description of \Cref{definition:classification_points} we get that 
the only real point of~$\Moebius_{\infty}$ is~$\dinf$.
To fully unveil the information contained in~$\Moebius_{\infty}$
we then pass to the complex numbers. 
All the constructions we made so far are algebraic, 
and so they make sense also over~$\C$.
Hence from now on, all the points, projective spaces and varieties, and quadratic forms
--- unless otherwise stated ---
are considered over the complex numbers. 

\begin{remark}
\label{remark:complex_distance}
In particular, using the same argument as in \Cref{proposition:distance}, 
for any $p_1, p_2 \in \Moebius$ with $p_i = (x_i: y_1: z_i: r_i: 1)$
--- even with complex coordinates --- we have
\[
 - \frac{1}{2} \bigl[ (x_1 - x_2)^2 + (y_1 - y_2)^2 + (z_1 - z_2)^2 \bigr] =
  \sprod{p_1}{p_2} \,.
\]
\end{remark}

Due to \Cref{proposition:distance}, we can extend the (squared) distance in~$\R^3$
to a rational function $d \colon \p^4 \times \p^4 \dashrightarrow \p^1$, defined by
\[ 
	d(p_1,p_2) := ( \sprod{p_1}{p_2}:h_1 h_2) \,,
\]
where $p_i = (x_i: y_i: z_i: r_i: h_i)$ for $i \in \{1,2\}$.
Notice that, however, the value of $d(p_1, p_2)$ does not depend on the choice of the representatives.
If $d(p_1, p_2) = (1:0)$, we write $d(p_1, p_2) = \infty$. 
From the definition, one derives that $d(p_1, p_2)$ is not defined if and only if
\[
 \bigl(
  p_1 \in \Moebius_{\infty} \text{ and } p_2 \in \Tan_{p_1} \Moebius
 \bigr)
 \quad \text{or} \quad 
 \bigl(
  p_2 \in \Moebius_{\infty} \text{ and } p_1 \in \Tan_{p_2} \Moebius
 \bigr) \,,
\]
where $\Tan_{p} \Moebius$ is the \emph{embedded tangent space} of~$\Moebius$ at~$p$,
namely, if $p=(x':y':z':r':h')$, then
\[
	\Tan_{p}\Moebius = \{(x:y:z:r:h) \in \p^4 \, | \, 2(xx' + yy' + zz') - (rh' + hr') = 0\}\,.
\]
In fact, an immediate computation shows that
\[
 \sprod{p_1}{p_2} = 0
 \quad \iff \quad
 p_2 \in \Tan_{p_1} \Moebius
 \quad \iff \quad
 p_1 \in \Tan_{p_2} \Moebius \,.
\]

Here we provide a few results on the behavior of the map~$d$,
in particular when it is applied to points in~$\Moebius_{\infty}$.

A direct computation shows the first result.

\begin{lemma}
	\label{lemma:distance_finite_dinf}
	If $p \in \Moebius$ is finite, then $d(p,\dinf) = \infty$.
\end{lemma}

\begin{definition}
\label{definition:dinf}
For simple infinite points, we define a map
\[
 \col \colon \Moebius_{\infty} \setminus \{ \dinf \} \longrightarrow A, 
 \qquad
 (x:y:z:r:0) \mapsto (x:y:z) \,.
\]
\end{definition}

\begin{lemma}
\label{lemma:different_color_infty}
	If $p_1, p_2 \in \Moebius$ are simple infinite and $\col(p_1) \ne \col(p_2)$, then $d(p_1,p_2)=\infty$.
\end{lemma}
\begin{proof}
    Since $p_1,p_2 \in \Moebius_{\infty}$, either $d(p_1,p_2)$ is undefined, or $d(p_1,p_2)=\infty$, 
    and the latter happens precisely when $\sprod{p_1}{p_2} \neq 0$.
    Write $p_i=(x_i:y_i:z_i:1:0)$.
	Let $B_i := \{(x,y,z) \in \C^3 \, | \, x_i\,x + y_i\,y + z_i\,z = 0 \}$. 
	From $p_i \in \Moebius$ we get $x_i^2+y_i^2+z_i^2 = 0$, namely, $(x_i,y_i,z_i) \in B_i$.
	On the other hand, since $\col(p_1) \ne \col(p_2)$,
	we have that $(x_1,y_1,z_1)$ and $(x_2,y_2,z_2)$ are linearly independent.
	Therefore the intersection $B_1 \cap B_2$ is one-dimensional
	and it cannot happen that $(x_1,y_1,z_1)$ and $(x_2,y_2,z_2)$ are both in $B_1 \cap B_2$,
	because otherwise they would be linearly dependent. 
	However, the statements $(x_2,y_2,z_2) \notin B_1$ and $(x_1,y_1,z_1) \notin B_2$
	are equivalent, so none of the two points belong to~$B_1 \cap B_2$.
	Hence,
	\[
		\sprod{p_1}{p_2} = x_1x_2 + y_1y_2 + z_1z_2 \neq 0 \,.
	\]
	This proves the statement.
\end{proof}

\begin{lemma}
\label{lemma:same_color_equal}
	Let $q_1$ and $q_2$ be simple infinite points such that $\col(q_1)=\col(q_2)$.
	If $p$ is finite, and both $d(q_1,p)$ and $d(q_2,p)$ are not $\infty$,
	then $q_1=q_2$ and $d(q_i,p)$ is undefined.
\end{lemma}
\begin{proof}
	Write $q_1=(x_1:y_1:z_1:r_1:0)$. Then $q_2=(x_2:y_2:z_2:r_2:0)=(\alpha x_1:\alpha y_1:\alpha z_1:r_2:0)$
	for some $\alpha \in \C\setminus \{0\}$ by the assumption $\col(q_1)=\col(q_2)$.
	We want to show that $r_2 = \alpha r_1$.
	Let $p=(x:y:z:r:1)$. Then
	\[
		d(p,q_i)=(\sprod{p}{q_i} : 0) = (x_ix + y_iy + z_iz -\frac{1}{2}r_i : 0)\,.
	\]
	Since both the values $d(p,q_i)$ are not $\infty$ by assumption, they are both undefined.
	Hence, $r_i=2(x_ix + y_iy + z_iz)$ and we get
	\[
		r_2=2(x_2x + y_2y + z_2z) = 2\alpha (x_1x + y_1y + z_1z) = \alpha r_1\,,
	\]
	which shows that $q_1=q_2$.
\end{proof}

We conclude this section with a key technical result, 
showing that the function~$d$, once restricted to particular linear subsets,
behaves like a 1-dimensional distance function.

\begin{definition}
	\label{definition:finp}
	If $p$ is simple infinite, we define~$\Fin_p$ to be
	the set of all finite points~$p'$ such that $d(p,p') \neq \infty$,
	namely $d(p,p')$ is undefined.
	Hence $\Fin_p = \Moebius \cap \Tan_{p} \Moebius \cap \{h \neq 0\}$.
\end{definition}

\begin{lemma}
\label{lemma:restriction_distance}
	Let $p$ be a simple infinite point.
    There exists a function $\pi \colon \Fin_p \longrightarrow \C$ such that
	for all $q_1,q_2 \in \Fin_p$, we have 
	\[
	  d(q_1, q_2) = \Bigl( \bigl( \pi(q_1)-\pi(q_2) \bigr)^2 : 1 \Bigr) \,.
    \]
\end{lemma}
\begin{proof}
 Write $p = (\tilde{x}: \tilde{y}: \tilde{z}: \tilde{r}: 0)$. 
 By possibly re-labeling the coordinates, we can suppose that $\tilde{x} \neq 0$,
 since $p$ is different from $\dinf = (0:0:0:1:0)$.
 The embedded tangent space~$\Tan_{p} \Moebius$ has equation
 \[
  2\tilde{x} \, x + 2\tilde{y} \, y + 2\tilde{z} \, z - \tilde{r} \, h = 0 \,.
 \]
 Therefore for any $q_1, q_2 \in \Fin_p$ we get, once we write $q_i = (x_i: y_i: z_i: r_i: 1)$
 \[
  \tilde{r} = 2\tilde{x} \, x_1 + 2\tilde{y} \, y_1 + 2\tilde{z} \, z_1 = 
  2\tilde{x} \, x_2  + 2\tilde{y} \, y_2 + 2\tilde{z} \, z_2 \,.
 \]
 Since we suppose that $\tilde{x} \neq 0$, we can write
 \[
  x_i = \frac{\tilde{r}}{2\tilde{x}} - \frac{\tilde{y}}{\tilde{x}} y_i - \frac{\tilde{z}}{\tilde{x}} z_i \,.
 \]
 \Cref{remark:complex_distance} shows that
 \begin{align*}
  -2 \sprod{q_1}{q_2} &= (x_1 - x_2)^2 + (y_1 - y_2)^2 + (z_1 - z_2)^2 \\
   &= \left(- \frac{\tilde{y}}{\tilde{x}} y_1 - \frac{\tilde{z}}{\tilde{x}} z_1 + \frac{\tilde{y}}{\tilde{x}} y_2 + \frac{\tilde{z}}{\tilde{x}} z_2\right)^2 + (y_1 - y_2)^2 + (z_1 - z_2)^2 \\
   &= \frac{1}{\tilde{x}^2} \left[ \bigl( \tilde{y}(y_2 - y_1) + \tilde{z}(z_2 - z_1) \bigr)^2 + \tilde{x}^2 (y_1 - y_2)^2 + \tilde{x}^2 (z_1 - z_2)^2 \right] \\
   &= \frac{1}{\tilde{x}^2} \left[ 2 \tilde{y} \tilde{z} (y_2 - y_1)(z_2 - z_1) + (\tilde{x}^2 + \tilde{y}^2)(y_1 - y_2)^2 + (\tilde{x}^2 + \tilde{z}^2)(z_1 - z_2)^2 \right] \\
   &= -\frac{1}{\tilde{x}^2} \left[ \tilde{z}(y_1 - y_2) - \tilde{y}(z_1 - z_2) \right]^2 ,
 \end{align*}
 where in the last step we use that $\tilde{x}^2 + \tilde{y}^2 + \tilde{z}^2 =0$. 
 Define the function $\pi \colon \Fin_p \longrightarrow \C$ by 
 \[
  (x:y:z:r:1) \mapsto \frac{1}{\sqrt{2}\tilde{x}}(\tilde{z} y - \tilde{y}z) \,.
 \]
 Notice that $\pi$ does not depend on the chosen representative of~$p$.
 The statement then follows, since
 \[
  d(q_1, q_2) = \bigl( \sprod{q_1}{q_2} : 1 \bigr) =  \Bigl( \bigl( \pi(q_1)-\pi(q_2) \bigr)^2 : 1 \Bigr) \,. \qedhere
 \]
\end{proof}

\section{Colorings and zero-sum cycles}
\label{result}

This section is devoted to the proof of the main result of the paper (\Cref{theorem:main}).
Let us first formalize the notion of flexibility.
Recall that a \emph{triangular polyhedron} is 
a finite two-dimensional abstract simplicial complex 
such that every edge belongs to exactly two faces.
The \emph{1-skeleton} of a triangular polyhedron is the graph defined by
the vertices and edges of the polyhedron.
We consider only polyhedra whose 1-skeleton is connected.

\begin{definition}
\label{definition:realization}
	A \emph{realization} of a triangular polyhedron whose 1-skeleton is $G=(V,E)$
	is a map $\rho \colon V \longrightarrow \R^3$ such that $\rho(u) \neq \rho(v)$
	for every $\{u, v\} \in E$.
	The realization $\rho$ induces \emph{edge lengths} $\lambda = (\lambda_e)_{e \in E}$ where 
	$\lambda_{\{u,v\}} := \left\| \rho(u) - \rho(v) \right\| \in \R_{>0}$ for $\{u,v\} \in E$.
	Here $\left\| \cdot \right\|$ is the standard Euclidean norm.
	Two realizations~$\rho_1$ and~$\rho_2$ are called \emph{congruent} 
	if there exists an isometry~$\sigma$ of~$\R^3$ such that $\rho_1 = \sigma \circ \rho_2$.

	 A \emph{flex} of the graph $G$ with a realization~$\rho$ is a continuous map
	 $f \colon [0, 1) \longrightarrow (\R^3)^V$ such that 
	 \begin{itemize}
	 \item
	  $f(0)$ is the given realization~$\rho$;
	 \item
	  for any $t \in [0,1)$, the realizations $f(t)$ and $f(0)$ induce the same edge lengths;
	 \item
	  for any two distinct $t_1, t_2 \in [0,1)$, the realizations~$f(t_1)$ and~$f(t_2)$ are not congruent.
	 \end{itemize}
\end{definition}

To prove the main result, the first step is to convert the information on the existence of a flex
for a triangular polyhedron into a combinatorial object, 
namely a coloring of the vertices of the polyhedron (\Cref{definition:coloring}).
This is done by considering special ``limits'' of the images of realizations in the flex under the conformal embedding,
which exist due to the fact that~$\Moebius$ is compact (\Cref{definition:limit_realization}).
These limits are not realizations, because some of the vertices
are sent to points in~$\Moebius$ that do not correspond to points in~$\R^3$;
these points are, actually, not even given by real coordinates. 
The properties of this coloring follow from the results of \Cref{compactification}.
Then, by arguing purely combinatorially we derive the existence
of a monochromatic cycle in the 1-skeleton of the polyhedron that passes 
through the edge whose dihedral angle is supposed to change along the flex
(\Cref{lemma:cycle}).
At this point, we notice that the vertices in this cycle
satisfy the hypotheses of \Cref{lemma:restriction_distance} (\Cref{lemma:red}). 
Then the function~$d$ behaves like a 1-dimensional distance function,
for which the main result is thus trivial.

We start by precisely stating our main result.

\begin{theorem}
\label{theorem:main}
    Let $G$ be the 1-skeleton of a triangular polyhedron with a realization that admits a flex.
    Let $\lambda$ be the edge lengths induced by the realizations in the flex.
	Let $\{ \vertOne, \vertTwo \}$ be an edge of~$G$ and let $\South$ and $\North$ be the two opposite vertices of the two triangles containing~$\{ \vertOne, \vertTwo \}$.
	If the dihedral angle
	between the faces $\{\vertOne,\vertTwo,\South\}$ and $\{\vertOne,\vertTwo,\North\}$ is not constant along the flex,
	then there is an induced\footnote{Given a subset~$V'$ of the vertices~$V$ of a graph~$G$, the \emph{induced subgraph determined by~$V'$} is the subgraph of~$G$ with vertices~$V'$ and all edges of~$G$ having both vertices in~$V'$. A subgraph is called \emph{induced} if it is the induced subgraph determined by a subset of vertices.} cycle in $G$ containing~$\{ \vertOne, \vertTwo \}$ but neither the vertex~$\South$ nor~$\North$ and
	there is a sign assignment such that the signed sum of the edge lengths~$\lambda_e$ over the edges~$e$ in the cycle is zero.
\end{theorem}

Let us start laying the foundation of the proof.
For the rest of this section, we fix the skeleton~$G = (V,E)$, the flex,
the edge~$\{ \vertOne, \vertTwo \}$, the edge lengths~$\lambda$ and
the two vertices~$\South$ and~$\North$ satisfying the assumptions of \Cref{theorem:main}.
Notice that the realizations in the flex of the two triangles $\{\vertOne,\vertTwo,\South\}$ and $\{\vertOne,\vertTwo,\North\}$ are non-degenerate
-- i.e., the vertices are not collinear ---
because otherwise their dihedral angle is not defined. 

The Zariski closure in~$(\C^3)^V$ of the set of real realizations inducing~$\lambda$ is an algebraic set
and we denote it by~$W$. 
By construction, $W$ is the set of maps $\rho \colon V \longrightarrow \C^3$ 
such that for all $\{v_1, v_2\} \in E$ with $\rho(v_i) = (x_i, y_i, z_i)$ we have
\begin{equation}
\label{equation:W}
\tag{$\ast$}
 (x_1 - x_2)^2 + (y_1 - y_2)^2 + (z_1 - z_2)^2 = \lambda^2_{\{v_1, v_2\}} \,.
\end{equation}
Notice that, as soon as a real realization~$\rho$ belongs to~$W$,
then all real realizations congruent to~$\rho$ belong to~$W$.
To effectively use the hypothesis of having a flex, 
we need to get rid of this abundance of ``copies'' of a single realization.
To do so, we take a ``slice'' of~$W$, which has the effect to kill
the action of the group of isometries by fixing a triangle.
Let $\rho_0$ be the realization of the flex at time~$0$.
Consider the subset
\[
 Z := 
 \bigl\{
  \rho \in W \, | \,
  \rho(\vertOne) = \rho_0(\vertOne), \;
  \rho(\vertTwo) = \rho_0(\vertTwo), \;
  \rho(\North) = \rho_0(\North)
 \bigr\}
 \,. 
\]
By construction, no two different elements of~$Z$ coming from real realizations may be congruent by a direct isometry;
they may be congruent by the reflection along the plane of the triangle with vertices~$\rho_0(\vertOne)$, $\rho_0(\vertTwo)$, and~$\rho_0(\North)$, 
but this does not influence our argument.
Now, the realizations in the flex may not be elements of~$Z$, 
but for each realization in the flex there is a congruent realization in~$Z$.
The image of~$Z$ via the product of conformal embeddings $\C^3 \hookrightarrow \Moebius$ is a subset of~$\Moebius^V$.

\begin{definition}
\label{definition:closure}
	Let $\variety$ be the Zariski closure in~$M^V$ of the image of~$Z$ 
	under the product of conformal embeddings $\C^3 \hookrightarrow \Moebius$. 
\end{definition}

Notice that the elements of~$\variety$ are maps 
$\rho \colon V \longrightarrow \Moebius$ which need not to correspond to realizations,
since the image of some vertices may lie on~$\Moebius_{\infty}$ or may have complex coordinates.
Consider the projection~$\variety_{\South}$ of~$\variety$ on the copy of~$\Moebius$ indexed by the vertex~$\South$.
By the assumption that the dihedral angle at $\{ \vertOne, \vertTwo \}$ changes during the flex, 
we get that $\variety_{\South}$ contains infinitely many points.
Hence $\variety_{\South}$ must intersect the hyperplane section $\{h=0\} \cap \Moebius = \Moebius_{\infty}$,
since it is a positive-dimensional projective subvariety of~$M$.

\begin{definition}
\label{definition:limit_realization}
From the previous discussion we know that we can 
pick an element~$\rho_{\infty}$ in~$\variety$ such that $\rho_{\infty}(\South) \in \Moebius_{\infty}$.
We fix such an element for the rest of the section.
\end{definition}

Notice that, by construction of~$Z$, 
all three points~$\rho_{\infty}(\vertOne)$, $\rho_{\infty}(\vertTwo)$, and~$\rho_{\infty}(\North)$ are finite.
The latter property is crucial for our argument;
if we just wanted to achieve $\rho_{\infty}(\South) \in \Moebius_{\infty}$
we could have used the set~$W$ without the need of introducing~$Z$ 
but this would not be enough to prove the main result of the paper.

\begin{lemma}
	\label{lemma:adjacentVertices}
	If $v$ and $w$ are adjacent vertices in~$G$, 
	then $d\bigl(\rho_{\infty}(v), \rho_{\infty}(w)\bigr)$ is either finite or undefined.
\end{lemma}
\begin{proof}
  The map~$d \bigl(\cdot(v), \cdot(w)\bigr) \colon \variety \dashrightarrow \p^1$ is a rational map
	that is constant on the dense subset of~$\variety$ of finite real points.
	In fact, these points are indeed realizations and since $v$ and $w$ are adjacent their distance in~$\R^3$ is constant, 
	which implies, due to \Cref{proposition:distance}, that $d$ is constant on them.
	The constant equals $(-\frac{1}{2}\lambda_{\{v,w\}}^2:1)$.
	Since $d(\cdot(v), \cdot(w))$ is a rational map, it is constant on~$\variety$ wherever it is defined,
	and as we showed this constant is different from~$\infty$.
\end{proof}

From \Cref{lemma:adjacentVertices} we obtain that $\rho_{\infty}(\South)$ cannot be~$\dinf$,
because otherwise we would have $d(\rho_{\infty}(\South), \rho_{\infty}(\vertOne)) = \infty$ due to \Cref{lemma:distance_finite_dinf}.

\begin{definition}
\label{definition:coloring}
We color the vertices of~$G$ with three colors:
\[
	\text{a vertex } v \in V \text { is }
	\begin{cases}
		\text{red}   & \text{if $\rho_{\infty}(v)$ is finite}, \\
		\text{blue}  & \text{if $\rho_{\infty}(v)$ is simple infinite and } \col\bigl(\rho_{\infty}(v)\bigr)=\col\bigl(\rho_{\infty}(\South)\bigr)\,,\\
		\text{gold} & \text{otherwise}.
	\end{cases} 
\]
\end{definition}

\begin{lemma}
	\label{lemma:no3}
	A gold vertex cannot be adjacent to a blue and a red vertex simultaneously.
 	In particular, there is no triangle with three colors.
\end{lemma}
\begin{proof}
	For a contradiction, let $\{u,w\}$ and $\{v,w\}$ be edges such that $u$ is red, $v$ is blue, and $w$ is gold.
	There are two cases:
	\begin{itemize}
	  \item $\rho_{\infty}(w) = \dinf$: since $\rho_{\infty}(u)$ is finite and $u$ is adjacent to $w$,
	  	by \Cref{lemma:adjacentVertices} and \Cref{lemma:distance_finite_dinf} we get a contradiction.
	  \item $\rho_{\infty}(w)$ is simple infinite and $\col\bigl(\rho_{\infty}(w)\bigr) \neq \col\bigl(\rho_{\infty}(\South)\bigr)$:
        however by assumption we have $ \col\bigl(\rho_{\infty}(\South)\bigr) = \col\bigl(\rho_{\infty}(v)\bigr)$,
        so $\col\bigl(\rho_{\infty}(w)\bigr) \neq \col\bigl(\rho_{\infty}(v)\bigr)$, and 
        this contradicts \Cref{lemma:adjacentVertices} and \Cref{lemma:different_color_infty}, since $v$ and $w$ are adjacent. \qedhere
	\end{itemize}
\end{proof}

Let $T$ be the set of triangles of~$G$ such that the vertices are only red and blue, and both colors occur.
Note that the triangle $\{ \vertOne, \vertTwo, \South \}$ belongs to~$T$ 
since $\vertOne, \vertTwo$ are red and~$\South$ is blue by construction;
instead, the triangle $\{ \vertOne, \vertTwo, \North \}$ does not belong to~$T$,
since all its vertices are red by construction.
We define a graph~$G_T$ with vertex set~$T$, and two triangles are adjacent
if they share an edge whose vertices have different colors (see Figure~\ref{figure:walk}).  
By \Cref{lemma:no3} and the fact that every edge is in exactly two triangles,
all vertices of~$G_T$ have degree two. Hence, the graph~$G_T$ is a disjoint union of cycles.
 
Let $\cycle$ be the cycle in $G_T$ containing the triangle $\{ \vertOne, \vertTwo, \South\}$.
Let $T_\cycle$ be the vertices of~$\cycle$ (i.\,e., triangles) 
and let $V' := \bigcup_{t \in T_\cycle} t$ be the set of vertices of~$G$ appearing in the triangles of $\cycle$.
The subgraph of~$G$ induced by the red vertices in~$V'$
defines a closed walk\footnote{A \emph{walk} is a finite sequence of vertices such that consecutive vertices are adjacent (i.\,e., edges or vertices might repeat). \emph{Closed} means that the first and last vertex are the same.} in~$G$, namely,
a triangle in~$\cycle$ contributes to the walk by an edge if it has two red vertices,
and just by a vertex otherwise.
The edge~$\{\vertOne,\vertTwo\}$ is contained in the walk.
\begin{definition}
	\label{definition:walks}
	The walk above is called the \emph{red walk (in $G$ obtained via $\rho_\infty$)}.
	Similarly, the walk obtained using blue vertices is called the \emph{blue walk}.
	Figure~\ref{figure:walk} portraits the situation. 
\end{definition}

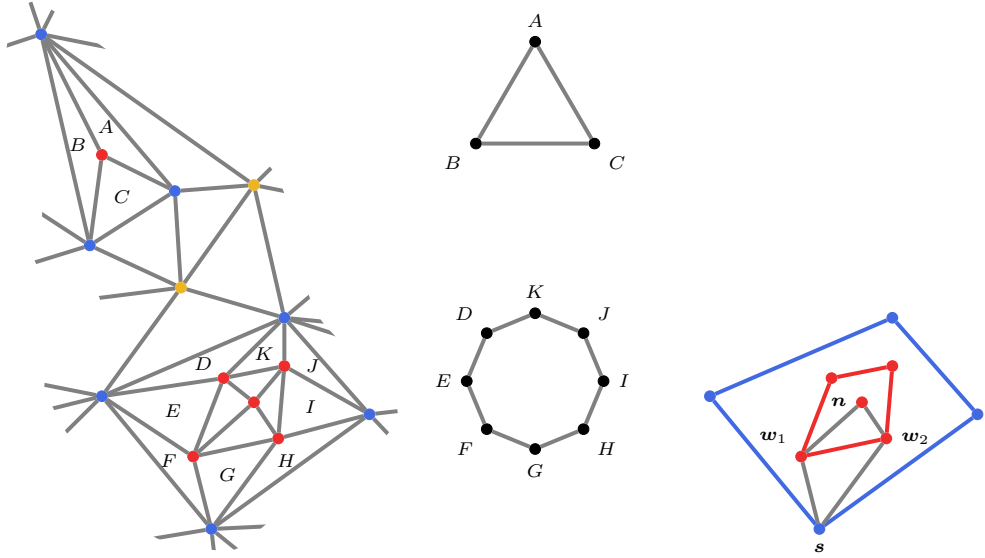
\begin{figure}[ht]
  \centering
  \begin{tikzpicture}[]
    \begin{scope}[scale=0.8,yscale=-1]
			\node[vertex,colB] (1) at (0,0) {};
			\node[vertex,colR] (2) at (1,2) {};
			\node[vertex,colB] (3) at (2.2,2.6) {};
			\node[vertex,colB] (4) at (0.8,3.5) {};
			\node[vertex,colG] (5) at (2.3,4.2) {};
			\node[vertex,colG] (6) at (3.5,2.5) {};
			\node[vertex,colB] (7) at (4,4.7) {};
			\node[vertex,colB] (8) at (1,6) {};
			\node[vertex,colR] (9) at (3,5.7) {};
			\node[vertex,colR] (10) at (4,5.5) {};
			\node[vertex,colR] (11) at (2.5,7) {};
			\node[vertex,colB] (12) at (2.8,8.2) {};
			\node[vertex,colR] (13) at (3.5,6.1) {};
			\node[vertex,colR] (14) at (3.9,6.7) {};
			\node[vertex,colB] (15) at (5.4,6.3) {};
			\coordinate (a1) at (3,10);
			\coordinate (a14) at (5,10);
			\coordinate (a3) at (7,8);
			\coordinate (a4) at (8,6);
			\coordinate (a5) at (7,5);
			\coordinate (a6) at (6,3);
			\coordinate (a7) at (5,1);
			\coordinate (a8) at (1,-2);
			\coordinate (a9) at (-1,-3);
			\coordinate (a10) at (-3,1);
			\coordinate (a11) at (-4,5);
			\coordinate (a12) at (-3,7);
			\coordinate (a13) at (-2,9);

			\clip[rounded corners] ($(1)!0.2!(a7)$) -- ($(1)!0.2!(a8)$) -- ($(1)!0.2!(a9)$) -- ($(1)!0.2!(a10)$) -- ($(4)!0.2!(a10)$) -- ($(4)!0.2!(a11)$) -- ($(5)!0.2!(a11)$) -- ($(8)!0.2!(a11)$) -- 
				($(8)!0.2!(a12)$) -- ($(8)!0.2!(a13)$) -- ($(12)!0.2!(a13)$) -- ($(12)!0.2!(a1)$) -- ($(12)!0.2!(a14)$) -- ($(12)!0.2!(a3)$) -- ($(15)!0.2!(a3)$) -- ($(15)!0.2!(a4)$) -- ($(7)!0.2!(a4)$) -- 
				($(7)!0.2!(a5)$) -- ($(7)!0.2!(a6)$) -- ($(6)!0.2!(a6)$) -- ($(6)!0.2!(a7)$) -- cycle;
			\draw[edge] (1)edge(2) (1)edge(3) (1)edge(4) (1)edge(6) (1)edge(a7) (1)edge(a8) (1)edge(a9) (1)edge(a10) (2)edge(3) (2)edge(4) (3)edge(4) (3)edge(5) (3)edge(6)
			(4)edge(5) (4)edge(a10) (4)edge(a11) (5)edge(6) (5)edge(7) (5)edge(8) (5)edge(a11) (6)edge(7) (6)edge(a6) (6)edge(a7) (7)edge(8) (7)edge(9) (7)edge(10) 
			(7)edge(a4)
			(7)edge(a5) (7)edge(a6) (8)edge(9) (8)edge(11) (8)edge(12) (8)edge(a11) (8)edge(a12) (8)edge(a13) (9)edge(10) (9)edge(11) 
			(11)edge(12)
			(12)edge(a1) (12)edge(a14) (12)edge(a3) (12)edge(a13) (15)edge(a3) (15)edge(a4);
			\draw[edge] (9)edge(13) (10)edge(13) (11)edge(13) (10)edge(14) (11)edge(14) (12)edge(14) (7)edge(15) (10)edge(15) (12)edge(15) (13)edge(14) (14)edge(15);
			
			\node[labelsty] at ($1/3*(1)+1/3*(2)+1/3*(3)$) {$A$};
			\node[labelsty] at ($1/3*(1)+1/3*(2)+1/3*(4)$) {$B$};
			\node[labelsty] at ($1/3*(2)+1/3*(3)+1/3*(4)$) {$C$};
			\node[labelsty] at ($1/3*(7)+1/3*(8)+1/3*(9)$) {$D$};
			\node[labelsty] at ($1/3*(8)+1/3*(9)+1/3*(11)$) {$E$};
			\node[labelsty] at ($1/3*(8)+1/3*(11)+1/3*(12)$) {$F$};
			\node[labelsty] at ($1/3*(11)+1/3*(12)+1/3*(14)$) {$G$};
			\node[labelsty] at ($1/3*(12)+1/3*(14)+1/3*(15)$) {$H$};
			\node[labelsty] at ($1/3*(10)+1/3*(14)+1/3*(15)$) {$I$};
			\node[labelsty] at ($1/3*(7)+1/3*(10)+1/3*(15)$) {$J$};
			\node[labelsty] at ($1/3*(7)+1/3*(9)+1/3*(10)$) {$K$};
    \end{scope}
    \begin{scope}[xshift=6.5cm,yshift=-1cm,scale=0.9]
			\coordinate (o) at (0,0);
      \node[vertex,label={[labelsty]above:$A$}] (A) at (0,1) {};
      \node[vertex,label={[labelsty]above:$B$},rotate around=120:(o)] (B) at (A) {};
      \node[vertex,label={[labelsty]above:$C$},rotate around=-120:(o)] (C) at (A) {};
      
      \coordinate (o2) at (0,-4);
      \node[vertex,label={[labelsty]right:$I$}] (I) at ($(o2)+(1,0)$) {};
      \node[vertex,label={[labelsty]right:$J$},rotate around=45:(o2)] (J) at (I) {};
      \node[vertex,label={[labelsty]right:$K$},rotate around=90:(o2)] (K) at (I) {};
      \node[vertex,label={[labelsty]right:$D$},rotate around=135:(o2)] (D) at (I) {};
      \node[vertex,label={[labelsty]right:$E$},rotate around=180:(o2)] (E) at (I) {};
      \node[vertex,label={[labelsty]right:$F$},rotate around=-135:(o2)] (F) at (I) {};
      \node[vertex,label={[labelsty]right:$G$},rotate around=-90:(o2)] (G) at (I) {};
      \node[vertex,label={[labelsty]right:$H$},rotate around=-45:(o2)] (H) at (I) {};
      
      \draw[edge] (A)edge(B) (B)edge(C) (C)edge(A);
      \draw[edge] (H)edge(I) (I)edge(J) (J)edge(K) (K)edge(D) (D)edge(E) (E)edge(F) (F)edge(G) (G)edge(H);
    \end{scope}
    \begin{scope}[xshift=8cm,scale=0.8),yscale=-1]
      \node[vertex,colB] (w7) at (4,4.7) {};
			\node[vertex,colB] (w8) at (1,6) {};
			\node[vertex,colR] (w9) at (3,5.7) {};
			\node[vertex,colR] (w10) at (4,5.5) {};
			\node[vertex,colR,label={[labelsty]above left:$\vertOne$}] (w11) at (2.5,7) {};
			\node[vertex,colB,label={[labelsty]below:$\South$}] (w12) at (2.8,8.2) {};
			\node[vertex,colR,label={[labelsty]right:$\vertTwo$}] (w14) at (3.9,6.7) {};
			\node[vertex,colR,label={[labelsty]left:$\North$}] (w16) at (3.5,6.1) {};
			\node[vertex,colB] (w15) at (5.4,6.3) {};
			\draw[edge,colR] (w9)edge(w10) (w10)edge(w14) (w14)edge(w11) (w11)edge(w9);
			\draw[edge,colB] (w7)edge(w8) (w8)edge(w12) (w12)edge(w15) (w15)edge(w7);
			\draw[edge] (w11)edge(w12) (w14)edge(w12) (w11)edge(w16) (w14)edge(w16);
    \end{scope}
  \end{tikzpicture}
  \caption{On the left, a portion of a triangular polyhedron,
  where we highlight the colors of the vertices and we label the triangles in $T$.
  In the center, the graph~$G_T$. On the right, the red and blue walks.
  We do not claim that these situations do indeed come from flexible polyhedra:
  this example serves solely to illustrate  the concepts.}
  \label{figure:walk}
\end{figure}

\begin{lemma}
	\label{lemma:blue}
	For all vertices~$v$ in the blue walk, we have $\rho_{\infty}(v)=\rho_{\infty}(\South)$.
\end{lemma}
\begin{proof}
	Let $v_1$ and $v_2$ be two adjacent vertices of the blue walk.
	Let $r$ be the third vertex of the triangle in~$T_\cycle$ containing~$v_1$ and~$v_2$.
	The vertex~$r$ must be red, therefore $\rho_{\infty}(r)$ is finite.
	For $i\in\{1,2\}$, the value $d\bigl(\rho_{\infty}(v_i),\rho_{\infty}(r)\bigr)$ cannot be~$\infty$ by \Cref{lemma:adjacentVertices}.
	Since $v_1$ and~$v_2$ are blue, $\col\bigl(\rho_{\infty}(v_1)\bigr)=\col\bigl(\rho_{\infty}(v_2)\bigr)$.
	Hence, $\rho_{\infty}(v_1)=\rho_{\infty}(v_2)$ by \Cref{lemma:same_color_equal}.
	The claim follows since the blue walk is connected and contains~$\South$.
\end{proof}

\begin{lemma}
	\label{lemma:red}
	For all vertices~$v$ in the red walk, $\rho_{\infty}(v)$ is contained in~$\Fin_{\rho_{\infty}(\South)}$.
\end{lemma}
\begin{proof}
	There exists an adjacent vertex~$w$ of~$v$ in the blue walk.
	By \Cref{lemma:blue}, $\rho_{\infty}(w)=\rho_{\infty}(\South)$.
	The value $d\bigl(\rho_{\infty}(v),\rho_{\infty}(w)\bigr)$ is not infinity by \Cref{lemma:adjacentVertices} and $\rho_{\infty}(v)$ is finite,
	therefore the point~$\rho_{\infty}(v)$ is in $\Fin_{\rho_{\infty}(\South)}$ by \Cref{definition:finp}.
\end{proof}

\begin{lemma}
	\label{lemma:cycle}
	There is a cycle such that it contains the edge $\{ \vertOne, \vertTwo \}$,
	its vertices are in the red walk and
	it is an induced subgraph of $G$.
\end{lemma}

\begin{proof}
	If the red walk, which is closed, contained the edge $\{ \vertOne, \vertTwo \}$ twice,
	then both triangles containing $\{ \vertOne, \vertTwo \}$ would belong to the cycle $\cycle$.
	But this is not possible since $\{ \vertOne, \vertTwo, \North \}$ is not an element of~$T$
	because all its vertices are red.
	Hence, there is a cycle in $G$ containing $\{ \vertOne, \vertTwo \}$ such that all its vertices are in the red walk.
	Among all such cycles, we take one with the minimum number of edges.
	This guarantees that the cycle is an induced subgraph.  
\end{proof}

Now we are ready to prove the main result.

\begin{proof}[{Proof of \Cref{theorem:main}}]
 We have to show the existence of a cycle passing through the edge $\{\vertOne,\vertTwo\}$ and of a sign assignment for its edges
 such that the signed sum of the lengths of the edges in the cycle is zero.
 For that, we employ the constructions introduced so far, namely:
 \begin{enumerate}\renewcommand{\theenumi}{(\alph{enumi})}\renewcommand{\labelenumi}{\theenumi}
   \item choosing $\rho_{\infty}$ to be the special element of~$\variety$ 
   		such that $\rho_{\infty}(\South) \in \Moebius_{\infty}$ from \Cref{definition:limit_realization},
   \item using $\rho_{\infty}$ to color the vertices of the polyhedron (\Cref{definition:coloring}),
   \item constructing the red walk for all of whose vertices $v$ 
   		we have $\rho_{\infty}(v) \in \Fin_{\rho_{\infty}(\South)}$ (\Cref{lemma:red}), 
   \item constructing an induced cycle $\mathcal{D}$ containing $\{\vertOne,\vertTwo\}$ and whose vertices are in the red walk  (\Cref{lemma:cycle}).
 \end{enumerate} 
 We write $\mathcal{D} := (v_1, \dotsc, v_k, v_{k+1} = v_1)$.
 Since for all vertices~$v_j$ in~$\mathcal{D}$ the point $\rho_{\infty}(v_j)$ is in~$\Fin_{\rho_{\infty}(\South)}$,
 we know from \Cref{lemma:restriction_distance} 
 that there exists a map $\pi \colon \Fin_{\rho_{\infty}(s)} \longrightarrow \C$ such that
 for all $j \in \{1, \dotsc, k\}$
 \[
  d \bigl( \rho_{\infty}(v_j), \rho_{\infty}(v_{j+1}) \bigr) = \Bigl( \bigl( \pi(\rho_{\infty}(v_j)) - \pi(\rho_{\infty}(v_{j+1})) \bigr)^2 : 1 \Bigr)\,.
 \]
 By construction of the set~$W$, see \Cref{equation:W}, and 
 taking into account \Cref{remark:complex_distance}, 
 for all elements $\rho \in \variety$ we have
 \[
  \lambda_{\{v_j, v_{j+1}\}}^2 = -2\sprod{\rho(v_j)}{\rho(v_{j+1})} \,,
 \]
 where we take the representatives for~$\rho(v_j)$ and~$\rho(v_{j+1})$ with $h$-coordinate equal to~$1$.
 Hence in particular
 \[
  \lambda_{\{v_j, v_{j+1}\}}^2 = 
  -2\sprod{\rho_{\infty}(v_j)}{\rho_{\infty}(v_{j+1})} = 
  -2 \bigl( \pi(\rho_{\infty}(v_j)) - \pi(\rho_{\infty}(v_{j+1})) \bigr)^2 \,.
 \]
 We conclude that 
 \[
  \lambda_{\{v_j, v_{j+1}\}} = \pm \, \imath \sqrt{2} \Bigl( \pi \bigl( \rho_{\infty}(v_j) \bigr) - \pi \bigl( \rho_{\infty} (v_{j+1}) \bigr) \Bigr) \,,
 \]
 where $\imath$ is the imaginary unit.
 Therefore we can choose integers $\eta_{j} \in \{1, -1\}$ such that 
 \[
  \sum_{j=1}^{k} \eta_j \, \lambda_{\{v_j, v_{j+1}\}}
 \]
 is a telescoping sum, which yields~$0$.
 
 Clearly, the cycle~$\mathcal{D}$ does not contain the vertex~$\South$ as it is blue.
 If the vertex $\North$ was in~$\mathcal{D}$,
 then the construction in \Cref{lemma:cycle} would imply that $\mathcal{D} = (\vertOne,\vertTwo,\North,\vertOne)$.
 But then the condition on the edge lengths of~$\mathcal{D}$ would imply that 
 the three vertices $\vertOne,\vertTwo,\North$ are collinear throughout the flex.
 This would contradict the assumption that the dihedral angle at $\{\vertOne,\vertTwo\}$ changes. 
 Thus the statement follows.
\end{proof}

Notice that in our previous works in which we study flexible realizations of graphs in the plane \cite{flexibleLabelings} and on the sphere~\cite{GGLSsphere},
a flex implies the existence of a combinatorial object --- namely, a special edge coloring --- called a \emph{NAC}, resp.\ \emph{NAP-coloring}.
On the other hand, the existence of a NAC, resp.\ NAP-coloring, for a graph provides a construction of a flex of the graph in the plane, resp. on the sphere,
though the obtained flexible realizations might be rather degenerate.
Exploiting the idea of so-called \emph{butterfly motions}  (see \Cref{figure:butterfly_bricard,figure:butterfly_complicated})
described, among other sources, in \cite[Section~6]{flexibleLabelings},
we obtain analogous results also for flexible polyhedra.
The combinatorial structure in this case is a separating cycle with a sign assignment to its edges.

\begin{proposition}
	\label{proposition:flex_construction}
	Let $G$ be the $1$-skeleton of a triangular polyhedron.
	Let $S$ be a cycle in~$G$ that separates the graph, namely,
	removing edges and vertices of~$S$ from~$G$ yields a disconnected graph.
	For any sign assignment to the edges of~$S$, not all having the same sign, the polyhedron admits 
	a realizations with a flex such that 
	the signed sum of the edge lengths induced by the realization in the cycle~$S$ is zero
	and the dihedral angles at all edges of~$S$ vary along the flex.
\end{proposition}
\begin{proof}
	Write $S=(v_1, \dots, v_k, v_{k+1}=v_1)$.
	We construct a realization~$\rho$ of the polyhedron as follows:
	set $\rho(v_j)=(x_j, 0, 0)$, where
	$x_j<x_{j+1}$ if $\{v_j,v_{j+1}\}$ has the positive sign and $x_j>x_{j+1}$ otherwise.
	The assumption that not all signs on the edges of~$S$ are the same
	guarantees that we can choose numbers~$\{x_j\}_{j=1}^{k}$ satisfying the previous requirements.
	The rest of the vertices of the graph is mapped to arbitrary points.
	The realization~$\rho$ has a flex since the vertices of different connected components of $G \setminus S$
	can rotate independently around the $x$-axis. Hence, the statement follows.
\end{proof}

We remark that if a polyhedron is not homeomorphic to a sphere,
then a cycle is not necessarily separating.
On the other hand, if the polyhedron is homeomorphic to a sphere, 
then any induced cycle with at least four edges separates the graph.

\begin{figure}[ht]
\centering
 \tikz[baseline={(0,0)}]{\node at (0,0) {\includegraphics[width=.31\textwidth, trim = {0 0.5cm 0 0.5cm}, clip]{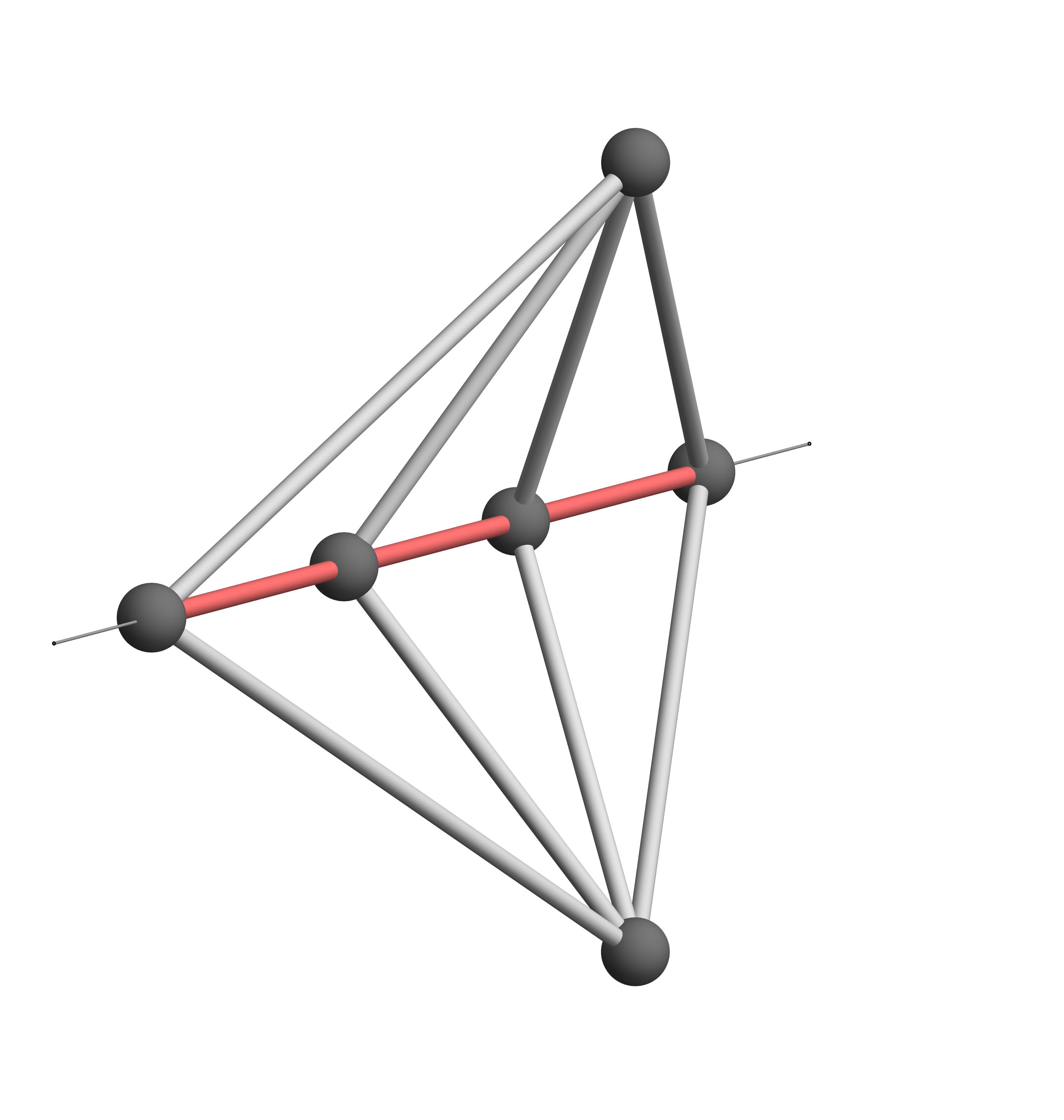}}}
 \caption{A butterfly motion of an octahedron. Notice the collinear realizations of vertices of a separating cycle, highlighted in red.}
 \label{figure:butterfly_bricard}
\end{figure}
\begin{figure}[H]
\centering
 \tikz[baseline={(0,0)}]{\node at (0,0) {\includegraphics[width=.18\textwidth, trim = {0 4.5cm 0 0.5cm}, clip]{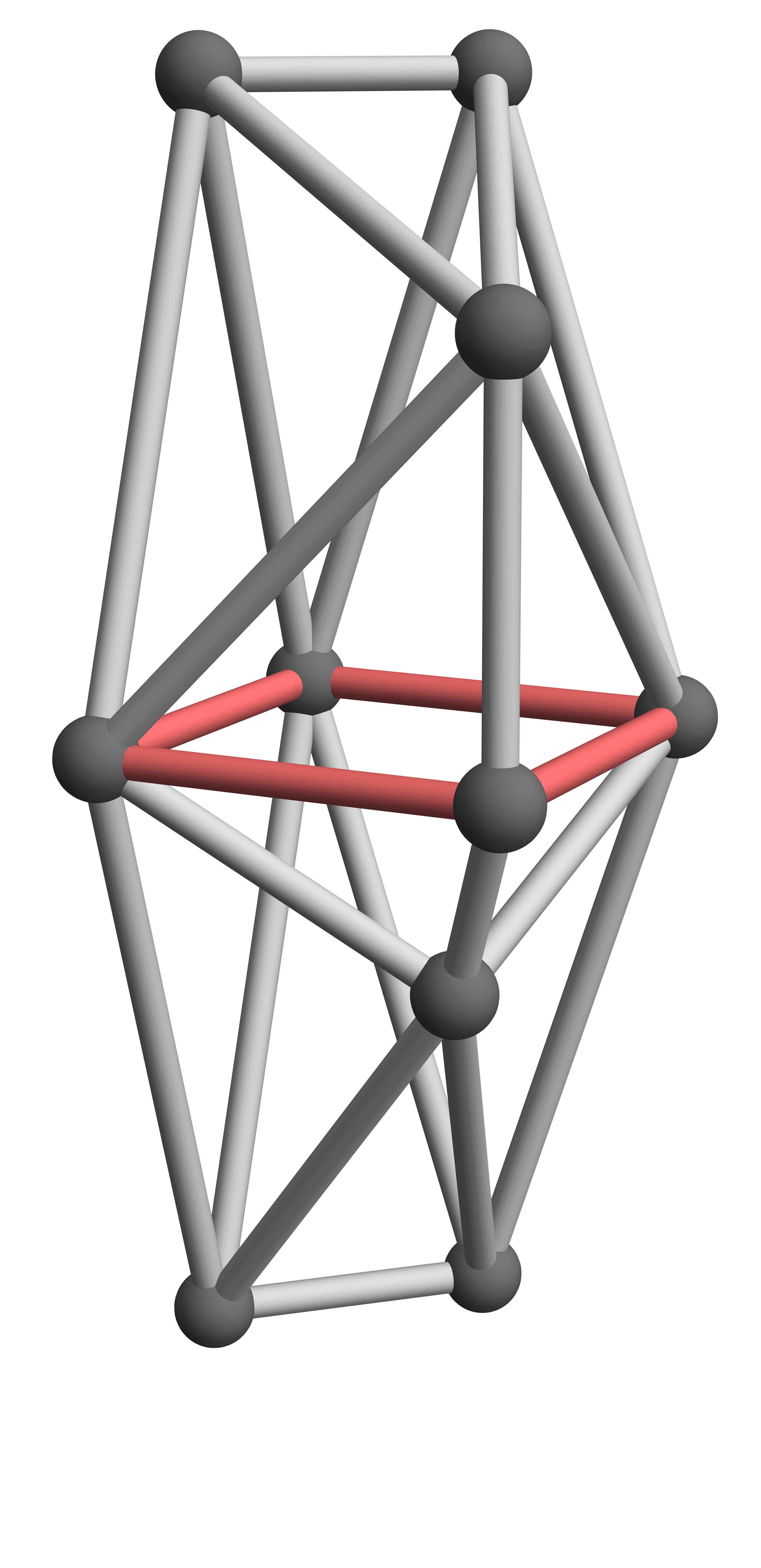}}}
 \qquad\qquad
 \tikz[baseline={(0,0)}]{\node at (0,0) {\includegraphics[width=.305\textwidth, trim = {0 1.5cm 2cm 1.75cm}, clip]{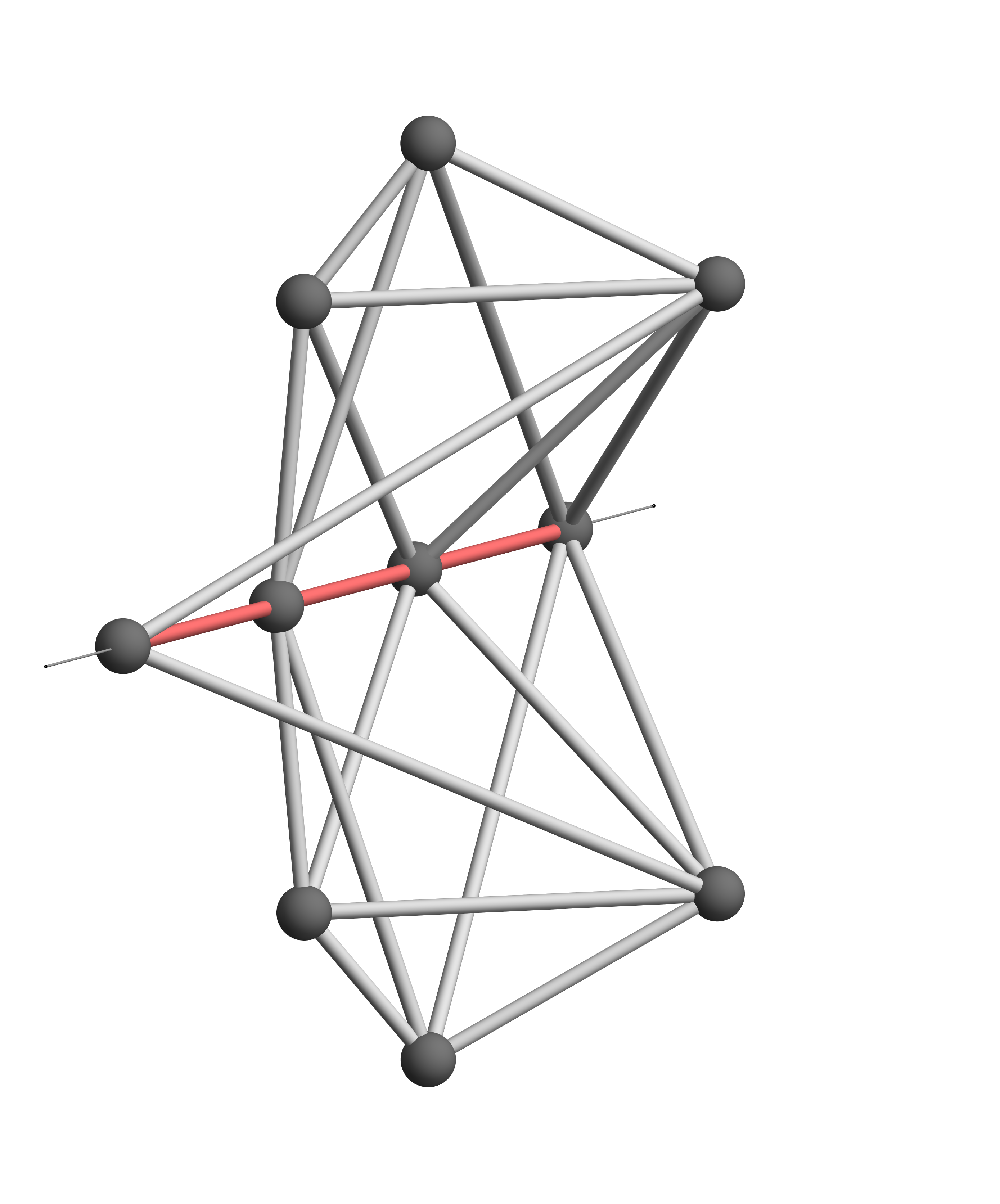}}}
 \tikz[baseline={(0,0)}]{\node at (0,0) {\includegraphics[width=.29\textwidth, trim = {2.25cm 1.5cm 3.5cm 1.75cm}, clip]{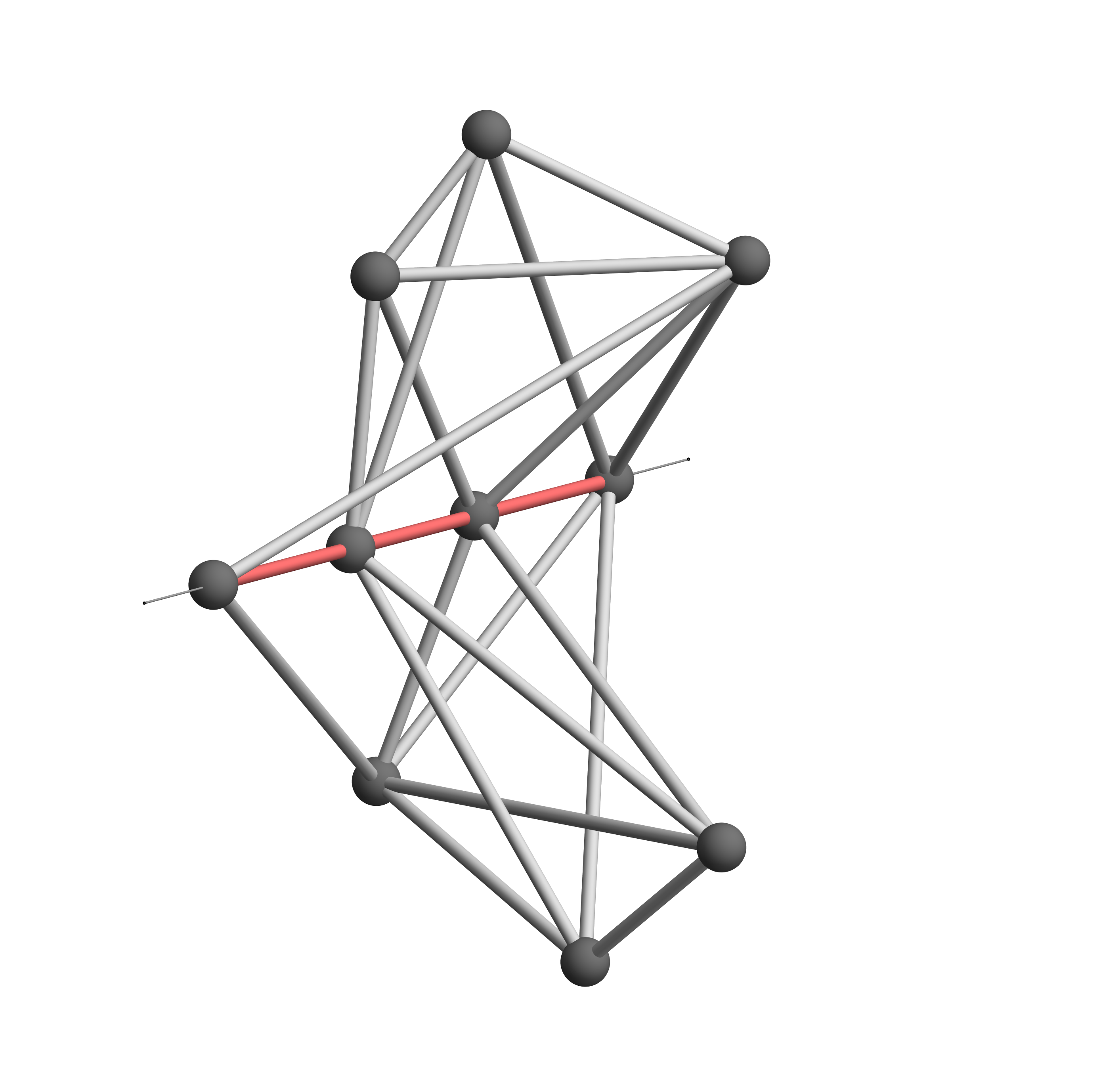}}}
 \caption{Instances, in a symmetric and a non-symmetric layout, of butterfly motions of a more complicated polyhedron (on the left).
 Notice the collinear realizations of vertices of a separating cycle, highlighted in red.}
 \label{figure:butterfly_complicated}
\end{figure}

\bigskip
\bigskip
\textsc{(MG) International School for Advanced Studies/Scuola Internazionale Superiore di
Studi Avanzati (ISAS/SISSA), Via Bonomea 265, 34136 Trieste, Italy}\\
Email address: \texttt{mgallet@sissa.it}

\textsc{(GG) Johann Radon Institute for Computational and Applied Mathematics (RICAM), Austrian
Academy of Sciences}\\
Email address: \texttt{georg.grasegger@ricam.oeaw.ac.at}

\textsc{(JL, JS) Johannes Kepler University Linz, Research Institute for Symbolic Computation (RISC)}\\
Email address: \texttt{jschicho@risc.jku.at}

\textsc{(JL) Department of Applied Mathematics, Faculty of Information Technology, Czech Technical University in Prague}\\
Email address: \texttt{jan.legersky@fit.cvut.cz}

\end{document}